\documentclass[11pt]{amsart}

\usepackage[ansinew]{inputenc}
\usepackage[a4paper,dvips]{geometry}
\usepackage{latexsym}
\usepackage{amsfonts}
\usepackage{amsmath,amssymb}
\usepackage{bbm}
\usepackage{color}
\usepackage{tikz}

\usepackage{graphicx}
\newtheorem{theorem}{Theorem}

\newtheorem{problem}{Problem}
\newtheorem{lemma}{Lemma}
\newtheorem{remark}{Remark}
\newtheorem{definition}{Definition}

\newcommand{\ZZ}{\mathbb{Z}}
\newcommand{\RR}{\mathbb{R}}

\newcommand{\mlattice}{\Pi_{N,\mathbf{v}}}
\newcommand{\mlatticee}{\Pi_{N,(a,b)}}

\newcommand{\mmlattice}{\Pi_{N,\mathbf{v}'}}

\newcommand{\mlatticeI}{\Pi^4_{N,\mathbf{v}}}
\newcommand{\mlatticeII}{\Pi^5_{N,\mathbf{v}}}

\newcommand{\covol}{\mathrm{vol}}

\title{{\large Long shortest vectors in low dimensional lattices}}

\author{Florian Pausinger}
\address{School of Mathematics \& Physics, Queen's University Belfast, BT7 1NN, Belfast, United Kingdom.}
\email{f.pausinger@qub.ac.uk}

\date{}

\begin{document}

\maketitle


\begin{abstract}
For coprime integers $N,a,b,c$, with $0<a<b<c<N$, we define the set
$$ \{ (na \! \! \! \! \pmod{N}, nb \! \! \! \! \pmod{N}, nc \! \! \! \! \pmod{N}) : 0 \leq n < N\}. $$
We study which parameters $N,a,b,c$ generate point sets with long shortest distances between the points of the set in dependence of $N$ and relate such sets to lattices of a particular form.
As a main result, we present an infinite family of such lattices with the property that the normalised norm of the shortest vector of each lattice converges to the square root of the Hermite constant $\gamma_3$. We obtain a similar result for the generalisation of our construction to $4$ and $5$ dimensions.
\end{abstract}

\section{Introduction}

We are interested in point sets
$$ \mlattice:=\{ (na \! \! \! \! \pmod{N}, nb \! \! \! \! \pmod{N}, nc \! \! \! \! \pmod{N}) : 0 \leq n < N\}, $$
in which $\mathbf{v} =(a,b,c)$ such that $0<a<b<c<N$ and $\gcd(N,a)=1$, which ensures that we have $N$ different points in the cube $[0,N-1]^3$.
Defining the shortest distance between points in $\mlattice$ as
$$ \lambda^{\ast}(\mlattice):= \min_{\substack{\mathbf{x} \neq\mathbf{y}, \\ \mathbf{x},\mathbf{y} \in \mlattice}} \| \mathbf{x} - \mathbf{y} \|, $$
we ask:  
\begin{enumerate}
\item[(Q1)] How long can the shortest distance between points of such sets be in dependence of $N$? 
\item[(Q2)] How to explicitly construct sets with long shortest distances?
\end{enumerate}
We start with simple bounds on $\lambda=\lambda^{\ast}(\mlattice)$. By the definition of our sets, the shortest possible distance between two points is $\sqrt{6}$ for a difference vector containing $1,-1,\pm 2$. On the other hand, assume that $\lambda$ is a shortest distance. Our point sets always contain $(0,0,0)$ and are always contained in the cube $[0,N-1]^3$. Therefore, a ball of radius $\lambda$ centered at $(0,0,0)$ (resp. its intersection with the cube $[0,N-1]^3$ which is $1/8$ of the ball) will never contain any other point of the set. In particular, we can attach such a fraction of a ball of radius $\lambda$ to any point in the set. This yields a simple upper bound on $\lambda$ since we know that our point set is contained in a cube of volume $(N-1)^3$ and we can attach a fraction of an empty ball of radius $\lambda$ to every point. Therefore,
\begin{equation} \label{heuristic1}
\sqrt{6} \leq \lambda^{\ast}(\mlattice) \leq C \cdot N^{2/3},
\end{equation}
for a constant $C >0$.

The main aim of this note is to answer (Q1) and (Q2). 
We provide an upper bound for the maximal, normalised shortest distance between two points in sets of the form $\mlattice$ and we relate these sets to particular lattices. 
Furthermore, in Theorem \ref{thm2} we present a parametrised family of such lattices whose normalised shortest distance converges (from below) to $C$ as the parameter goes to infinity. 

In the following we recall some basic notions and relate our point sets to lattices in $\RR^3$ before we state our main results and discuss previous work in Section \ref{sec:main}. Section \ref{sec:lattice} contains a detailed discussion of numerical results as well as of a remarkable lattice while Section \ref{sec:proofs} contains the proofs our main results.
Finally, in Section \ref{sec:multiD} we discuss the extension of our construction to dimensions $4$ and $5$.

\subsection{Lattices and reduced bases}
We introduce the main concepts surrounding lattices and refer to \cite{CS} for more information.
Let $\mathbf{V}=\{\mathbf{v}_1, \ldots, \mathbf{v}_d\}$ be linearly independent vectors in $\RR^d$. The lattice generated by $\mathbf{V}$ is the set
$$ \Lambda(\mathbf{V}) = \left\{ \sum_{i=1}^{d} x_i\cdot \mathbf{v}_i: x_i \in \ZZ \right\} = \{ \mathbf{V} \cdot \mathbf{x}: \mathbf{x} \in \ZZ^d \} $$
of all integer linear combinations of the vectors in $\mathbf{V}$, and the set $\mathbf{V}$ is called a basis for the lattice. Two bases $\mathbf{V}, \mathbf{V}'$ generate the same lattice, i.e. $\Lambda(\mathbf{V})=\Lambda(\mathbf{V'})$, if and only if there exists an integer unimodular matrix $\mathbf{U}$ such that $\mathbf{V}=\mathbf{V}' \mathbf{U}$.
This is also the reason why it makes sense to define the \emph{(co)volume} of a lattice, i.e.
$$ \covol(\Lambda(\mathbf{V})) = | \det M_\mathbf{V} |, $$
in which $M_\mathbf{V}$ is the square matrix whose rows are the basis vectors $\mathbf{v}_i$, $1\leq i \leq d$ of an arbitrary lattice basis $\mathbf{V}$.
For $\Lambda=\Lambda(\mathbf{V})$ and $1 \leq i \leq d$ the \emph{$i$-th successive minimum} $\lambda_i(\Lambda)$ is the radius of the smallest closed ball centred at the origin and containing at least $i$ linearly independent lattice vectors. In particular, $\lambda_1(\Lambda)$ is the length of the shortest vector of $\Lambda$.
Finally, the \emph{Hermite constant} $\gamma_d$ is (can be) defined as 
$$  \sqrt{\gamma_d}= \sup_{\Lambda} \frac{\lambda_1(\Lambda)}{\covol(\Lambda)^{1/d}},$$
in which the supremum is over all $d$-dimensional lattices $\Lambda$.
Thus, the Hermite constant is an upper bound on the maximal length of short lattice vectors in Euclidean space and is known only for $1\leq d \leq 8$ as well as $d=24$; in particular 
$$\gamma_2=2/\sqrt{3}, \ \ \gamma_3=2^{1/3}, \ \ \gamma_4=2^{1/2} \ \ \text{ and } \ \ \gamma_5=2^{3/5}.$$
Now we restrict to three dimensional lattices. 
\begin{definition} \label{reduced}
A basis $\mathbf{V}=\{\mathbf{v}_1, \mathbf{v}_2, \mathbf{v}_3 \}$ of a three dimensional lattice in $\RR^3$ is \emph{(Minkowski-) reduced} if its vectors satisfy:
\begin{enumerate}
\item[(B.1)] $\|\mathbf{v}_1 \| \leq \|\mathbf{v}_2 \| \leq  \|\mathbf{v}_3 \|$;
\item[(B.2)] $\|\mathbf{v}_2 + x_1 \mathbf{v}_1 \| \geq \| \mathbf{v}_2\| $ and $\| \mathbf{v}_3 + x_2 \mathbf{v}_2 + x_1 \mathbf{v}_1 \| \geq \| \mathbf{v}_3\| $ for all integers $x_1, x_2$.
\end{enumerate}
\end{definition}
We refer to \cite{semaev} for a three-dimensional lattice reduction algorithm and to \cite{stehle} for a characterisation of Minkowski-reduced lattice bases in arbitrary dimension $d$. Importantly, the shortest vector in a reduced basis is always the shortest vector of the lattice.
In fact, a basis of a $d$-dimensional lattice that reaches the $d$ minima must be Minkowski-reduced, but a Minkowski-reduced basis may not reach all the minima, except the first four (see \cite{waerden}).

\subsection{The underlying lattice of $\mlattice$}
A first important observation is that our point sets can be obtained as intersections of the cube $[0,N-1]^3$ with three-dimensional integer lattices and the shortest distance in our point sets is given by the shortest vector of the underlying lattices. (Note however that the shortest vector of the underlying lattice need not be an element of our sets!)
A second observation is that the shortest distance of a point set does not change when we change the order of the generating vector; i.e. $\lambda^{\ast}(\mlattice)=\lambda^{\ast}(\mmlattice)$ if $\mathbf{v}=(a,b,c)$ and $\mathbf{v}'=(a,c,b)$ or any other permutation of $a,b,c$. Hence, we can assume $a<b<c$.
In particular, our assumptions on $N,a,b,c$ allow us to restrict to generating vectors of the form $(1,b,c)$, since we can always solve the congruence 
\begin{equation} \label{restriction} a \cdot n' \equiv 1 \pmod{N}. \end{equation} 
Therefore, we can replace the generating vector $(a,b,c)$ with $(n'a,n'b,n'c)=(1,b',c')$ in which $n'$ is such that $n'a \equiv 1 \pmod{N}$, $n'b \equiv b' \pmod{N}$ and $n'c \equiv c' \pmod{N}$ generate the same set of points. Restricting to generating vectors of the form $(1,b,c)$ gives a particularly simple generating matrix for the underlying lattices.

\begin{lemma} \label{lem:basis}
Let $\mathbf{v}=(1,b,c)$ with $1 < b<c < N$. Then
\begin{equation} \label{modLattice} \mlattice = \begin{pmatrix} 0 & 0 & 1 \\ 0 & N & b \\ N & 0 & c \end{pmatrix} \ZZ^3 \cap [0,N-1]^3 \end{equation}
\end{lemma}

\begin{proof}
Note that we can always find unique integers $x,y$ such that $ 0<Nx+nb < N $ and $ 0<Ny+nc < N $ for $1\leq n \leq N-1$. Every non-trivial point in $\mlattice$ can thus be obtained by multiplying the matrix with $(y,x,n)$. On the other hand it is easy to see that these are the only vectors in $\ZZ^3\setminus\{\mathbf{0}\}$ that yield points in $[0,N-1]^3$. Finally, multiplying with $(0,0,0)$ gives $(0,0,0)$.
\end{proof}

For a given point set $\mlattice$, write $\Lambda(\mlattice)$ for the underlying lattice.
Lemma \ref{lem:basis} can be used to conclude that the shortest distance in our point sets is given by the shortest vector of the underlying lattices. 

\begin{lemma}\label{lem:short}
In the above notation, we have that $ \lambda^{\ast}(\mlattice) = \lambda_1(\Lambda(\mlattice)). $
\end{lemma}
\begin{proof}
By Lemma \ref{lem:basis} every point in $\mlattice$ can be written as an \emph{integer} linear combination of the vectors of a \emph{reduced} basis $\{\mathbf{v}_1, \mathbf{v}_2, \mathbf{v}_3\}$ of the underlying lattice. We identify a point  $\sum_{i=1}^3 x_i \mathbf{v}_i$ in $\mlattice$ with its coordinates $(x_1,x_2,x_3)$ with respect to the reduced basis. If we fix the first coordinate, i.e. set $x_1=k$ for $k \in \ZZ$, we see that all such points $(k,x_2,x_3)$ lie on the hyperplane spanned by $\mathbf{v}_2$ and $\mathbf{v}_3$ and translated by the constant vector $k \cdot \mathbf{v}_1$.

We assume for contradiction that for every point $(x_1,x_2,x_3)$ in $\mlattice$ the corresponding pair of points $(x_1 \pm 1, x_2, x_3)$ lies outside of $[0,N-1]^3$. In other words, that there is no pair of points of $\mlattice$ whose difference is $\mathbf{v}_1$.
Since the cube $[0,N-1]^3$ (interpreted as a compact, solid set) has no holes, this assumption implies that for fixed $(x_1,x_2,x_3)$ also points $(x_1 \pm k, x_2, x_3)$ for $k \in \ZZ \setminus \{ -1,0,1\}$ lie outside of the cube. Now take an arbitrary point $(x_1,x_2,x_3)$ in $\mlattice$. Due to its lattice structure $\mlattice$ contains all integer points of the intersection of $[0,N-1]^3$ with the hyperplane spanned by $\mathbf{v}_2$ and $\mathbf{v}_3$ and translated by $x_1 \cdot \mathbf{v}_1$. However, for all such points and, again because the cube has no holes, $\mlattice$ does not contain any point on parallel hyperplanes of the form $(x_1 \pm k, x_2, x_3)$ for arbitrary but fixed $k \in \ZZ \setminus \{ 0\}$.
This implies that all points of $\mlattice$ lie on a hyperplane spanned by the two longer vectors of the reduced basis and potentially translated by a constant multiple of $\mathbf{v}_1$. Since $\mlattice$ always contains $(0,0,0)$ we see that we actually have $x_1=0$.

A simple area argument (similar to the one in the introduction) shows that if all $N$ points of $\mlattice$ lie on the intersection of the cube with a single hyperplane, then $\|\mathbf{v}_2\|$ is $\mathcal{O}(\sqrt{N})$. However, this contradicts the assumption that $(x_1 \pm 1, x_2, x_3)$ lie outside of the cube for every $(x_1,x_2,x_3)$ in $\mlattice$.

To see this, we first note that a ball of radius $ \|\mathbf{v}_2\|/2$ centered at the centroid of the intersection of the hyperplane with the cube has to contain at least one point $(x_1,x_2,x_3)$ of $\mlattice$ lying on the hyperplane. Next we center at this point a second ball of radius $\|\mathbf{v}_2\|$ which has to contain the points $(x_1 \pm 1, x_2, x_3)$, one on each side of the plane, because $\mathbf{v}_2$ is at least as long as $\mathbf{v}_1$. This second ball is contained in a third ball centred again at the centroid of the plane and of radius $3\|\mathbf{v}_2\|/2$. We call the two halves of the ball on opposite sides of the plane \emph{hemispheres}.
Assuming that the centroid of the plane is at the center of the cube shows that the distance to all faces of the cube is at least $(N-1)/2$ and hence the third ball is fully contained in the cube (for $N$ large enough) which contradicts our assumption.
In case that the centroid is not at the center of the cube, it can be shown by a (a)symmetry argument that at least one of the hemispheres has to be fully contained in the cube, giving the same conclusion.
\end{proof}

\section{Main results}
\label{sec:main}
\subsection{Maximal shortest distance in three-dimensions}
We use the connection to lattices and their shortest vectors to refine the upper bound on the longest shortest distance. 
We start with an intuitive argument. It is well known and was first shown by Gauss in 1831 that the face-centered cubic (FCC) lattice
\begin{equation} \label{fcc}  \Omega:= \begin{pmatrix} 1 & 0 & 1 \\ 1 & 1 & 0 \\ 0& 1 &1 \end{pmatrix} \ZZ^3 \end{equation}
gives the densest (sphere) lattice packing in $\RR^3$. In fact, the FCC lattice gives the densest packing among all lattices according to the famous \emph{Kepler Conjecture} which was finally settled by Hales and Ferguson \cite{hales, lagarias} utilising an approach of Fejes T\'{o}th \cite{FT} as well as state-of-the-art formal proof techniques \cite{hales2}; see also \cite{szpiro}. Importantly, the shortest vector in the FCC lattice is exactly twice the radius of the spheres which are centered at the nodes of the lattice. We can refine the volume argument of the previous section to get an accurate value for the constant in the upper bound.

We recall that the the FCC lattice can be built out of the basic cube in Figure \ref{fig:basicCube}. 
Ignoring boundary effects, each of the $N$ points is contained in $8$ basic cubes and each basic cube contains $4$ lattice points. Thus, we can build $2N$ basic cubes with edge length $\ell$. 
We are looking for the maximal edge length $\ell$ such that all $2N$ basic cubes fit into $[0,N-1]^3$:
$$ \max_{\ell \in \RR} \ 2N \ell^3 \leq (N-1)^3 \ \ \ \ \Rightarrow \ \ \ \ \ell \approx \frac{N^{2/3}}{\sqrt[3]{2}} . $$
The shortest distances between the nodes of this lattice are the face diagonals. Thus, we get an accurate idea of the constant $C$ in \eqref{heuristic1}:
$$ C \approx \frac{1}{\sqrt[3]{2}} \sqrt{2} = 2^{1/6} \approx 1.122.$$

This argument can be made exact using the notion of the \emph{Hermite constant} of a lattice as well as Lemma \ref{lem:basis} and \ref{lem:short}. From Lemma \ref{lem:basis} we deduce that the covolume of our particular lattices is $N^2$. Hence,
$$ \frac{(\lambda^{\ast}(\mlattice))^2}{(N^2)^{2/3}} \leq \gamma_3 = 2^{1/3}  \ \ \Rightarrow \ \  \lambda^{\ast}(\mlattice) \leq 2^{1/6} N^{2/3}.$$

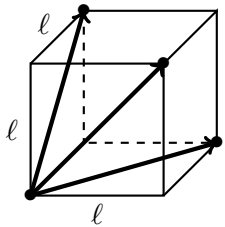
\begin{figure}[h!]
\begin{center}
\begin{tikzpicture}[scale=0.35]

\node at (0,0) {$\bullet$};
\node at (5,5) {$\bullet$};
\node at (7,2) {$\bullet$};
\node at (2,7) {$\bullet$};

\draw[thick] (0,0) -- (5,0) -- (5,5)--(0,5) --(0,0);
\draw[thick] (5,0) -- (7,2) -- (7,7) -- (5,5);
\draw[thick] (0,5) -- (2,7)--(7,7);
\draw[thick,dashed] (0,0) -- (2,2)--(7,2);
\draw[thick,dashed] (2,7)--(2,2);

\draw[ultra thick,->] (0,0)--(2,7);
\draw[ultra thick,->] (0,0)--(7,2);
\draw[ultra thick,->] (0,0)--(5,5);

\node at (2.5,-0.7) {$\ell$};
\node at (-0.7,2.5) {$\ell$};
\node at (0.5,6.5) {$\ell$};


\end{tikzpicture}
\end{center}
\caption{Basic building block of FCC lattice.} \label{fig:basicCube}
\end{figure}

\subsection{The two-dimensional case}
In \cite{paus2} the two-dimensional variant of the problem was studied. It is well known that the hexagonal lattice
\begin{equation} \label{hexLattice} \Lambda_h = \begin{pmatrix} 1 & 1/2 \\ 0 & \frac{\sqrt{3}}{2} \end{pmatrix} \ZZ^2  
\end{equation}
maximises the packing density in $\RR^2$; see \cite{fuk11,rogers58,rogers}. We define 
$$\mlatticee:=\{ (na \! \! \! \! \pmod{N}, nb \! \! \! \! \pmod{N}) : 0 \leq n < N\}$$ 
and $\lambda^{\ast}(\mlatticee)$ as in the three dimensional case.
It turns out that the shortest distance in any two-dimensional set $\mlatticee$ satisfies
$$\sqrt{2} \leq \lambda^{\ast}(\mlatticee) \leq \sqrt{N} \cdot \sqrt{\gamma_2}.$$ 
Set $n=4s+3$ and define
$$ \frac{b_s}{N_s}=[0,b_1,b_2, \ldots, b_n]=[0,2,1,2,1,\ldots, 1,2] $$
via this particular continued fraction expansion.
Note that this continued fraction converges to $(\sqrt{3}-1)/2$.
It was then shown that the particular family of integer lattices
$$ X_s=  \begin{pmatrix} 0 & 1 \\ N_s & b_s \end{pmatrix} \ZZ^2$$
converges to (a rotated and scaled version of) the hexagonal lattice as the parameter $s$ approaches infinity and the normalised length, $\lambda^{\ast}(\Pi_{N_s,(1,b_s)})/\sqrt{N}$, of the shortest distance converges to $\sqrt{\gamma_2}$.

\subsection{Main result and numerical observations}
Having the upper bound and knowing that the problem can be fully resolved in two dimensions motivates to look for optimal lattices in three dimensions.
As a main result we show how to construct lattices, $\Lambda(\mlattice)$, whose normalised shortest distance is arbitrarily close to $2^{1/6}$. 

\begin{theorem} \label{thm:main}
For every $\varepsilon >0$ there exist infinitely many pairs $(N,\mathbf{v})$ such that
$$\frac{\lambda^{\ast}(\mlattice)}{N^{2/3}} > 2^{1/6} - \varepsilon .$$
\end{theorem}
See Theorem \ref{thm2} in Section \ref{sec:proofs} for an explicit construction of such lattices and Theorem \ref{thm3} in Section \ref{sec:multiD} for a natural generalisation to $d=4$ and $d=5$. Our construction is based on extensive numerical analysis of all lattices, $\Lambda(\mlattice)$, up to $N=310$; see Figure \ref{fig:allN} (left). From these observations it appears at first as if it is not possible to obtain lattices with the longest possible shortest distances. The best lattice we can find up to $N=310$ is for the parameters $N=244$ with $\mathbf{v}=(1,13,169)$. We will study this remarkable lattice which led to our explicit construction in the next section.

\begin{figure}
\includegraphics[scale=0.5]{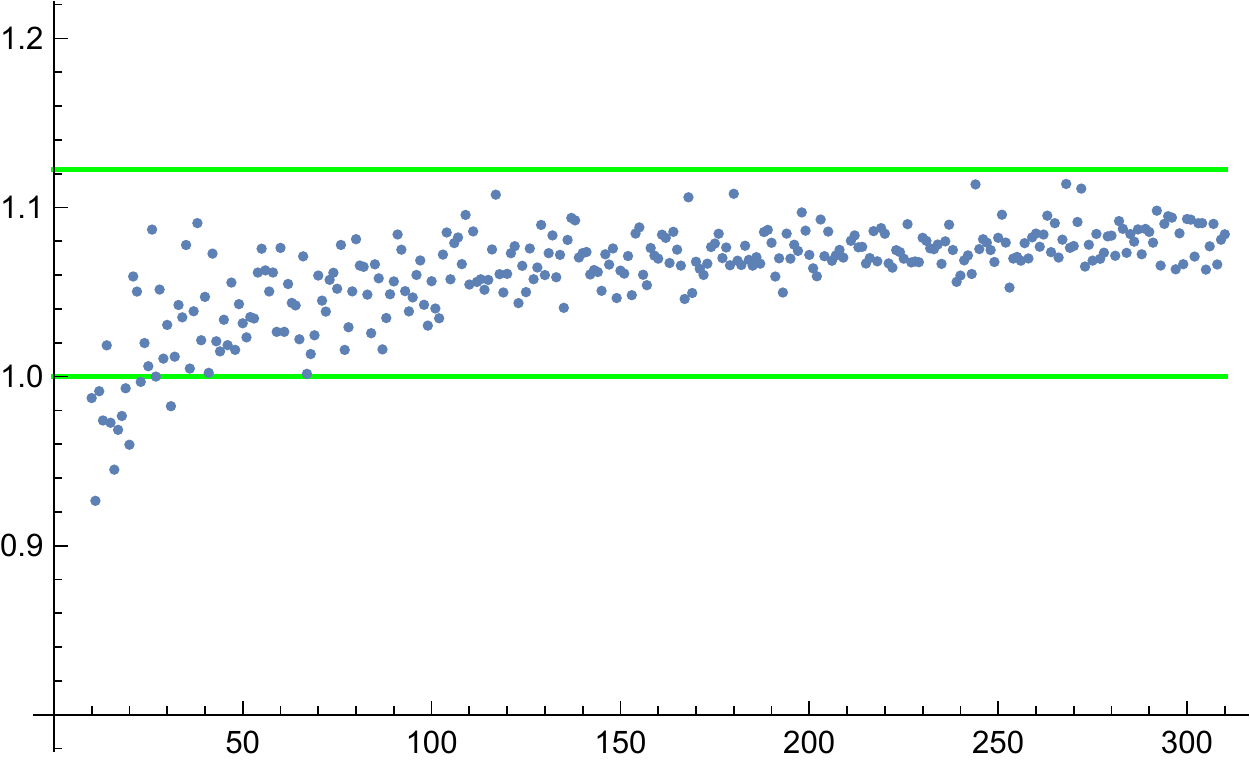} \ \ \ \ \
\includegraphics[scale=0.5]{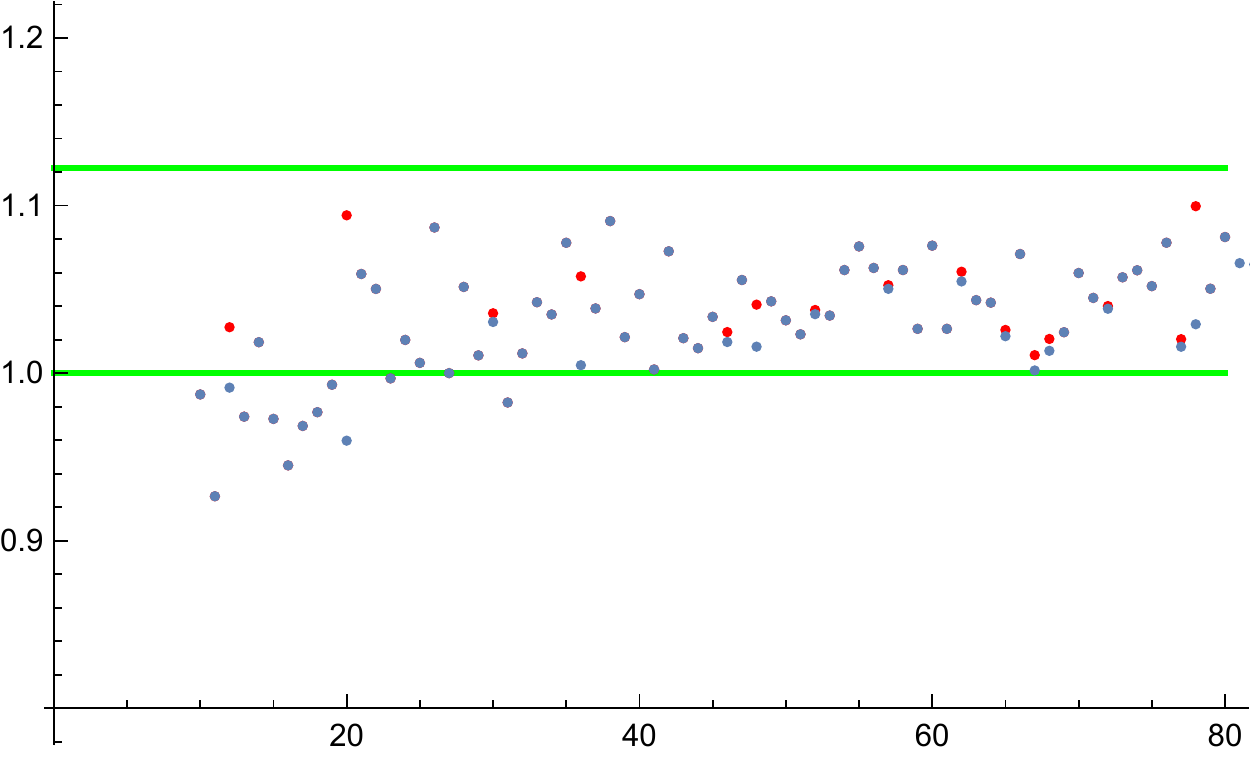}
\caption{\emph{Left:} The normalised length of the longest shortest vector in base $N$ among all lattices with generating vectors of the form $(1,b,c)$ with $1< b<c <N$. The green line at $2^{1/6}\approx 1.122$ illustrates the upper bound, while the second green line is at 1 and illustrates what can be achieved with cubic lattices. \emph{Right: } The red dots show improvements when allowing all generating vectors $(a,b,c)$ with $\gcd(N,a,b,c)=1$.} \label{fig:allN}
\end{figure}

Furthermore, we present a similar construction in Theorem \ref{thm1} of lattices such that the normalised shortest distance converges to 1. While the first family of lattices can be interpreted as an approximation of the FCC lattice, the second family approximates a cubic lattice.
To see this, set $N=n^3$ and look at the set
\begin{equation*} G_n= \begin{pmatrix} n^2 & 0 & 0 \\ 0 & n^2 & 0 \\ 0 & 0 & n^2 \end{pmatrix} \ZZ^3 \cap [0,N-1]^3. 
\end{equation*}
Then 
$$\frac{\lambda^{\ast}(G_n)}{N^{2/3}} = \frac{n^2}{n^2} = 1.$$
We close this section with an open problem for which we expect an affirmative answer.
\begin{problem}
Prove (or disprove) that there exists an absolute $\varepsilon >0$ such that for every $N>N_0$ there is a generating vector $\mathbf{v}=(1,b,c)$ with $1<b<c <N$ and
$$ \frac{\lambda^{\ast}(\mlattice)}{N^{2/3}} > 1+\varepsilon. $$
\end{problem}

\begin{remark}
If we relax the condition on the generating vector to $\gcd(N,a,b,c)=1$, then we can also get very good lattices for small $N$ as the examples $N=20$, $\mathbf{v}=(6,15,18)$ for which we get $\lambda^{\ast}(\mlattice)/N^{2/3} = 1.0942\ldots$ or $N=78$, $\mathbf{v}=(15,65,75)$ with $\lambda^{\ast}(\mlattice)/N^{2/3} = 1.09965\ldots$ show.
However, in general it seems that there is no systematic improvement as shown in Figure \ref{fig:allN} (right) at least not for $10\leq N \leq 80$.
\end{remark}

\section{A remarkable lattice}
\label{sec:lattice}

The lattice for $N=244$ and $\mathbf{v}=(1,13,169)$ has several remarkable properties which helped to find the family of lattices presented in Theorems \ref{thm1}-\ref{thm3} and which gives an idea why lattices with long shortest vectors seem so rare and hard to find.
First, we notice that $b^2=c$, i.e. $13^2=169$, and that $b\cdot c = b^3 = 2197 \equiv 1 \pmod{244}$. This property has important implications. If we look at the three two-dimensional orthogonal projections of the lattice, i.e. the lattices generated by
$$ 
\begin{pmatrix}  0 & 1 \\ 244 & 13 \end{pmatrix},
\begin{pmatrix}  0 & 1 \\ 244 &169 \end{pmatrix},
\begin{pmatrix} 0 & 13 \\ 244 & 169  \end{pmatrix},
$$
we see that they are all copies of the same two-dimensional lattice modulo $244$; see Figure \ref{fig:projections}.

\begin{figure}
\includegraphics[scale=0.35]{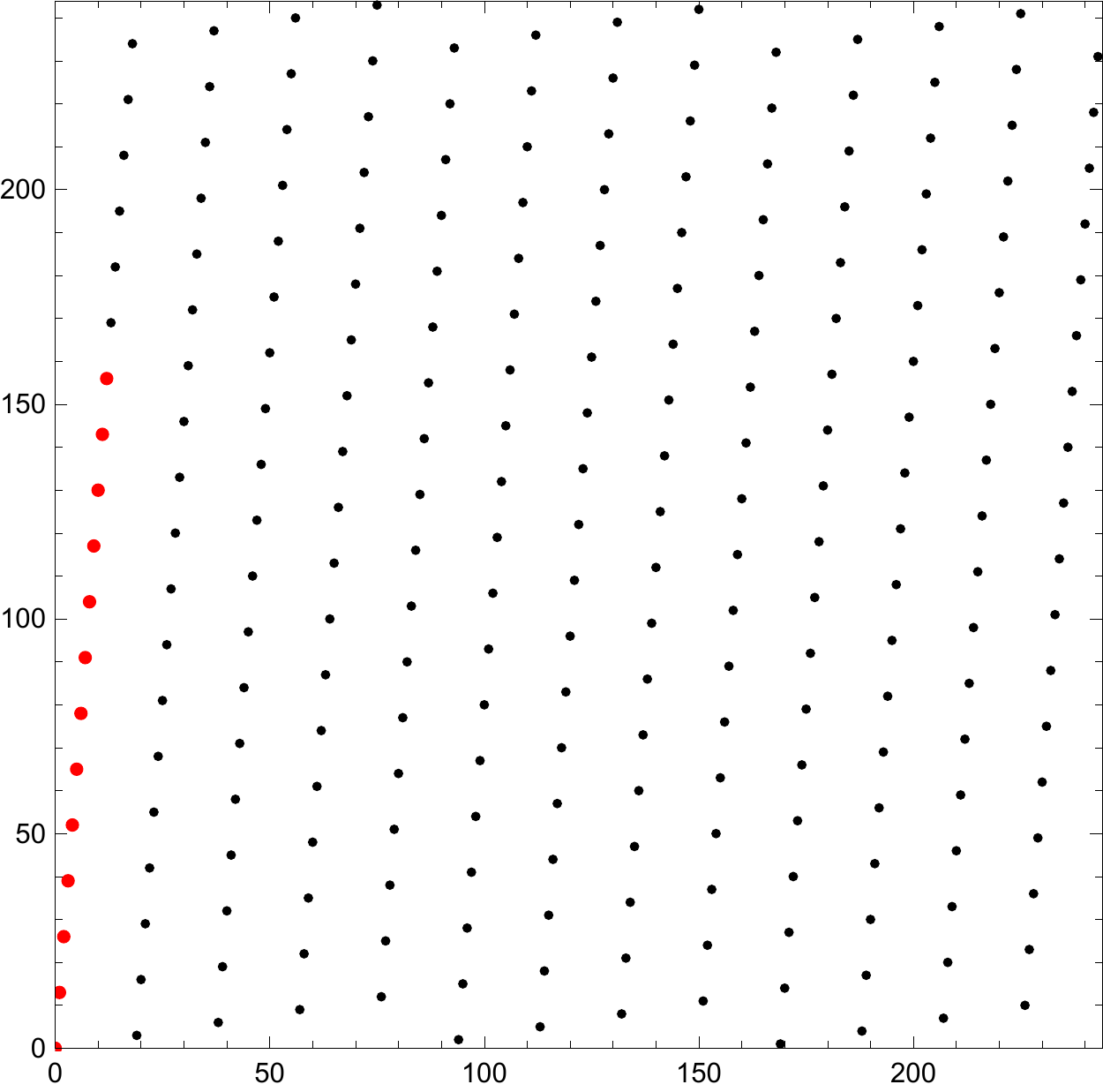} \quad \quad
\includegraphics[scale=0.35]{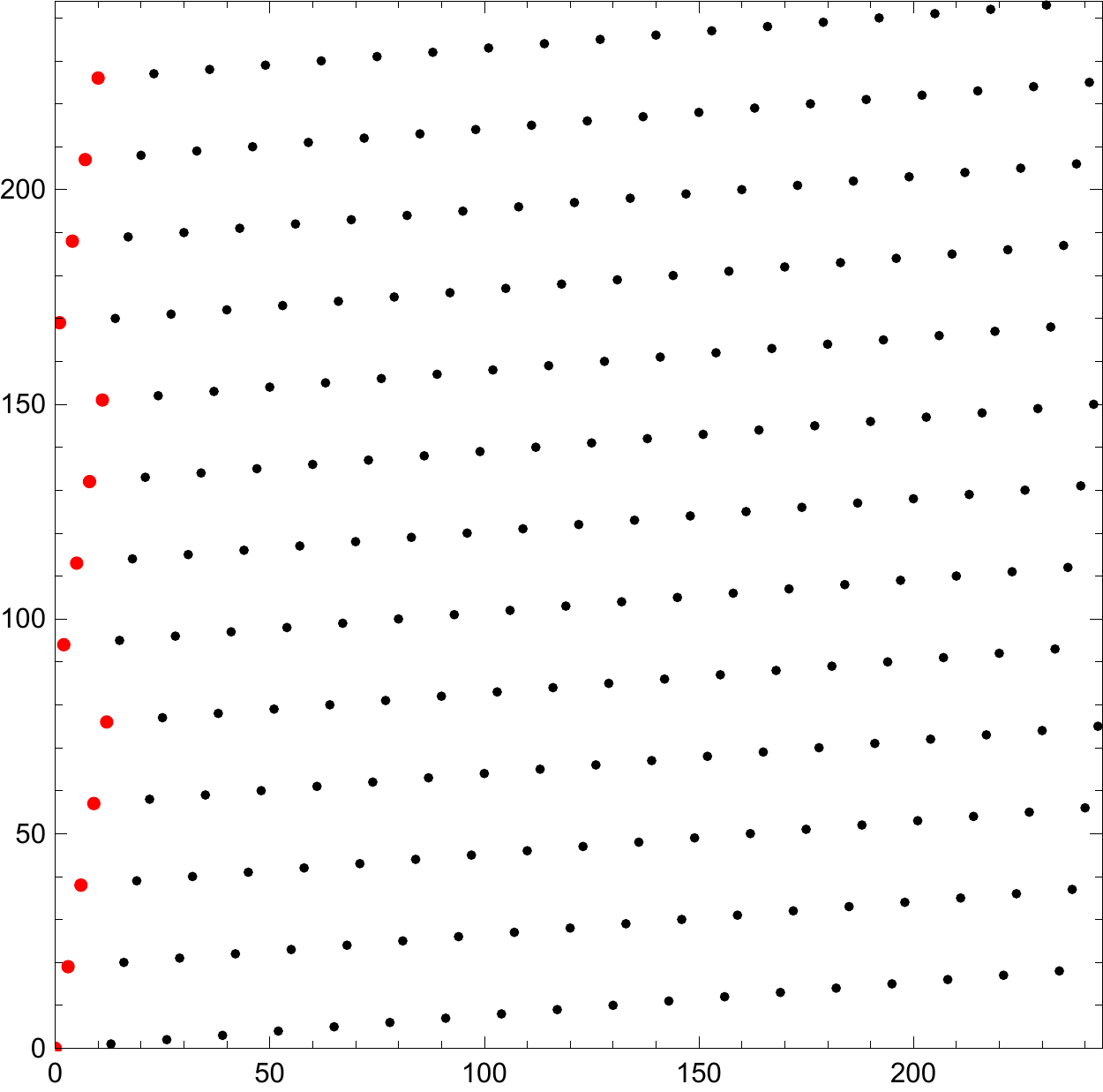} \quad \quad
\includegraphics[scale=0.35]{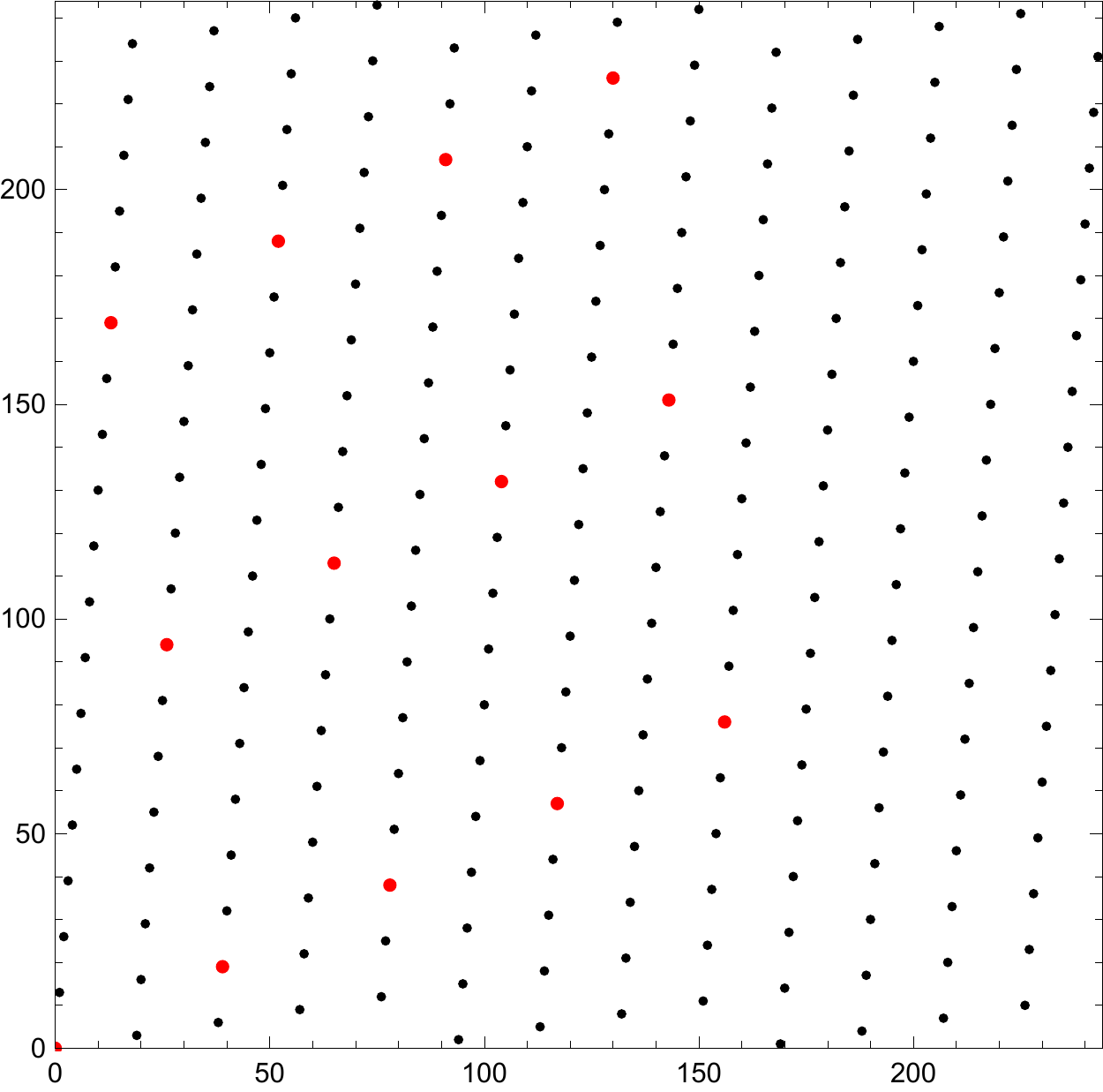} 
\caption{The three projections of the lattice $\Pi_{244,\mathbf{v}}$ with $\mathbf{v}=(1,13,169)$. The first $13$ points of each lattice are in red.} \label{fig:projections}
\end{figure}

The second lattice can be obtained from the first via a rotation by $\pi/2$ and a reflection on the $x$-axis:
$$ (1,b) \cdot  \begin{pmatrix}  0 & -1 \\ 1 & 0 \end{pmatrix} \cdot \begin{pmatrix}  1 & 0 \\ 0 & -1 \end{pmatrix} = (b,1) = (b\cdot 1, b \cdot c).$$
Thus, there is a bijection between the two lattices; the point $(n,bn)$ in the first lattice is obtained as point $(k,kc)$ for $k=bn$ in the second lattice.
Similarly, we see that there is also a bijection between the first and third lattice using both conditions on the parameters; i.e. $bc \equiv 1 \pmod {244}$ and $b^2 \equiv c \pmod {244}$. We have that
$$(1,b) \equiv (b^2 \cdot b, b^2 \cdot c) \pmod{244},$$
and hence the point $(n,bn)$ in the first lattice is obtained as point $(kb,kc)$ in the third lattice for $k=b^2 n$.
An easy calculation shows that the first of the three two-dimensional lattices has a reduced basis of the form $(1,13), (19,3)$; this corresponds to the points with indices $n=1$ and $n=19$.
Using the bijections of the indices of the points that we have just established, we can translate these two vectors and see that we obtain these vectors for indices $k=b=13$ and $k=19b=3$ in the second lattice as well as $k'=b^2=c=169$ and $k'=19b^2=-39$ in the third lattice. Interestingly, the three three-dimensional vectors for $n=1$, $k=b$, $k'=c$ give three variations of the vector $(1,b,c)$ thanks to our assumptions, whereas the vectors for $n=19$, $k=3$ and $k'=-39$ are
\begin{equation} \mathbf{v}_1=\begin{pmatrix} 19 \\ 3 \\ 39 \end{pmatrix}, 
\mathbf{v}_2=\begin{pmatrix} 3 \\ 39 \\ 19 \end{pmatrix}, 
\mathbf{v}_3=\begin{pmatrix} -39 \\ -19 \\ -3 \end{pmatrix} 
\end{equation}
and as such a reduced basis of the lattice.
Hence, we obtain
\begin{equation}\label{equ1} \frac{\lambda(\Pi_{244,(1,13,169)})}{244^{2/3}} = \frac{\sqrt{1891}}{244^{2/3}} = 1.11366\ldots \end{equation}
This lattice is a rhombohedral lattice which is an almost perfect approximation of the FCC lattice.
The angle $\alpha$ between the edges of the rhomboid is given as
$$ \alpha = \arccos \left( \frac{\mathbf{v}_1 \cdot \mathbf{v}_2}{\| \mathbf{v}_1\| \|\mathbf{v}_2 \|} \right) 
= \arccos \left( \frac{\mathbf{v}_1 \cdot \mathbf{v}_3}{\| \mathbf{v}_1\| \|\mathbf{v}_3 \|} \right)
=\arccos \left( \frac{\mathbf{v}_2 \cdot \mathbf{v}_3}{\| \mathbf{v}_2\| \|\mathbf{v}_3 \|} \right) = 1.06572 \ldots$$
which is almost the angle $\pi/3 = 1.0472\ldots$ obtained for the FCC lattice via

$$\arccos \left( \frac{(1,1,0) \cdot (1,0,1)}{\| (1,1,0)\| \|(1,0,1) \|} \right) = \pi/3 .$$

\subsection{A general procedure?}
The calculations in the previous section suggest the following general procedure:
Assume $N,b,c$ are such that $bc\equiv 1 \pmod{N}$ and $b^2 \equiv c \pmod{N}$ (or $bc\equiv -1 \pmod{N}$ and $b \equiv c^2 \pmod{N}$). 
Then it is easy to see, using the same index bijections as before, that the three two-dimensional projections of the lattice are indeed always variants of the same lattice.
And it is tempting to hope that the resulting lattice is again rhombohedral such that 
the vectors of the reduced basis are permutations of $$(\pm (x \pmod N), \pm (xb \pmod N), \pm (xb^2 \pmod N))$$ where
representatives of each class are taken from the interval $[-(N-1)/2, (N-1)/2]$ as in the example
for $N=244$ and $x$ being the non-trivial index of the shortest vector in the two-dimensional lattice. The length of the shortest vector should then be given as $\|(x,x b,x b^2)\|$.
It turns out, that the shortest vector is in general indeed of this form. Unfortunately, it also seems that in general only two of the three vectors of the reduced basis are of this form, while the third vector is (very) different. Thus, in general we do not obtain a rhombohedral lattice approximating the FCC lattice. Even worse, systematic experiments show that the third basis vector has a much larger norm in general, generating a lattice that is far from being rhombohedral.

As an example we look at $N=366$ and $\mathbf{v}=(1,13,169)$. We have that $b\cdot c = b^3 = 2197 \equiv 1 \pmod{366}$ and $b^2=c$. The three two-dimensional projections are variants of the lattice spanned by $(0,366)$ and $(1,13)$ with reduced basis $(1,13)$ and $(-28,2)$. Hence, in the above notation $x=-28$, $xb=-28 \cdot 13 \equiv 2 \pmod{366}$ and $xb^2 = -28 \cdot 169 \equiv 26 \pmod{366}$. 
However, we observe that 
\begin{equation} x+ xb + xb^2 = x(1 + b + b^2) \equiv 0 \pmod{366}\end{equation} 
-- a property which makes it impossible to get a reduced basis as in the case of $N=244$. This seems to be the generic case for examples of this kind. Indeed, the reduced basis of the three dimensional lattice $\Pi_{366, \mathbf{v}}$ is
\begin{equation} \mathbf{v}_1=\begin{pmatrix} -28 \\ 2 \\ 26 \end{pmatrix}, 
\mathbf{v}_2=\begin{pmatrix} -2 \\ -26 \\ 28 \end{pmatrix}, 
\mathbf{v}_3=\begin{pmatrix} 61 \\ 61 \\ 61 \end{pmatrix}.
\end{equation}
Hence, 
$$ \frac{\lambda(\Pi_{366, \mathbf{v}})}{366^{2/3}} = \frac{2\sqrt{366}}{366^{2/3}} = 0.7477\ldots, $$
and similarly for other lattices of this form and generated by $(1,13,169)$; see Table \ref{table1}.

\begin{table}[h]
\begin{center}
\begin{tabular}{|c|c|c|c|c|}
\hline
$N$ & $x$ & $bx$ & $bx^2$ &  $\lambda^{\ast}(\Pi_{N,\mathbf{v}})/N^{2/3}$\\
\hline
61 & -4 & 9 & -5& 0.7127\\
122 & 10& 8 & -18& 0.7830\\
183 & -14& 1& 13 & 0.5935\\
\hline
244 & 19 &3&39& 1.11366 \\
366 &-28&2&26& 0.7477\\
549 & -42 &3&39& 0.8560\\
\hline
732 &-56&4&52& 0.9421 \\
1098 &-84&6&78& 1.0785\\
$2196$ &-168&12&156& $1.0032^\ast$\\
\hline
\end{tabular}
\end{center}
\caption{Lattices satisfying the conditions $bc\equiv 1 \pmod{N}$ and $b^2 \equiv c \pmod{N}$ for $\mathbf{v}=(1,13,169)$. $^\ast$The shortest vector for $N=2196$ is given by $(1,13,169)$; see Theorem \ref{thm1}.}
\label{table1}
\end{table}

\section{Two infinite sets of lattices}
\label{sec:proofs}
The results in Table \ref{table1} and for similar sets of examples suggest that while the case of $N=244$ seems highly exceptional and mysterious there may be two general patterns for bases of the form $N=b^3-1$ and $N=(b^3-1)/2$ for integers $b>0$. We explore these examples in the following and note that Theorem \ref{thm2} implies Theorem \ref{thm:main}.
\begin{theorem} \label{thm1}
For an integer $m>2$ let $N=m^3-1$ and $\mathbf{v}=(1,m,m^2)$. Then 
$$\lambda^{\ast}(\Pi_{N,\mathbf{v}}) = \sqrt{1+m^2+m^4}$$
such that
$$ \lim_{m\rightarrow \infty} \frac{\lambda^{\ast}(\Pi_{N,\mathbf{v}})}{N^{2/3}} = 1. $$
\end{theorem}

\begin{proof}
We prove the theorem by showing that
\begin{equation} \mathbf{v}_1=\begin{pmatrix} 1 \\ m \\ m^2 \end{pmatrix}, 
\mathbf{v}_2=\begin{pmatrix} -m \\ -m^2 \\ -1 \end{pmatrix}, 
\mathbf{v}_3=\begin{pmatrix} -m^2 \\ -1 \\ -m \end{pmatrix} \end{equation}
is a reduced lattice basis for the lattice
\begin{equation*} \begin{pmatrix} 0 & 0 & 1 \\ 0 & m^3-1 & m \\ m^3-1 & 0 & m^2 \end{pmatrix}. \end{equation*}
First, we recall that two lattices with bases $\mathbf{B}$ and $\mathbf{B'}$ are equivalent if there exists a unimodular matrix $\mathbf{U}$ such that $\mathbf{B}\cdot \mathbf{U}=\mathbf{B'}$. We observe that
\begin{equation}
\begin{pmatrix} 1 &-m &-m^2\\ m&-m^2 &-1\\ m^2& -1 &-m\end{pmatrix}
\begin{pmatrix} m & 0&1\\ 1&-m &0\\ 0& 1 &0\end{pmatrix}
=\begin{pmatrix} 0 & 0 & 1 \\ 0 & m^3-1 & m \\ m^3-1 & 0 & m^2 \end{pmatrix}
\end{equation}
And the determinant of the second matrix is indeed 1.
Next, we show that $\mathbf{v}_1, \mathbf{v}_2, \mathbf{v}_3$ form a reduced basis. 
The first condition in Definition \ref{reduced} is obviously satisfied since all three vectors have the same set of entries. The second condition 
$$ \left \| \begin{pmatrix} -m \\ -m^2 \\ -1 \end{pmatrix} + x_1 \begin{pmatrix} 1 \\ m \\ m^2 \end{pmatrix} \right \| =
\left \| \begin{pmatrix} x_1 - m \\ x_1 m- m^2 \\ x_1 m^2 - 1 \end{pmatrix} \right \| 
\geq \left \| \begin{pmatrix} -m \\ -m^2 \\ -1 \end{pmatrix} \right \|$$
follows from the observation that
$$
1+m^2+m^4 + x_1^2 (m^4+m^2+1) -2x_1(m^3 + m^2 + m) \geq 1 + m^2 + m^4,
$$
for all $x_1\in \ZZ$ and $m>2$.
This is obviously true for all $x_1 \geq 2$ and all $x_1 \leq 0$. For $x_1=1$, we need that
\begin{equation} \label{cond1}
m^4+m^2+1 - 2m^3 - 2m^2 - 2m \geq 0,
\end{equation}
which holds for all integers $m>2$.
In a similar way, we can also verify that $\| \mathbf{v}_3 + x_2 \mathbf{v}_2 + x_1 \mathbf{v}_1\| \geq \|\mathbf{v}_3\|$. 
In this case the inequality reduces to
\begin{align*}
(m^4+m^2+1)(1 + x_1^2 +x_2^2) + (m^3+m^2 +m)(2x_1 - 2x_2 - 2x_1 x_2) \geq m^4+m^2+1.
\end{align*}
Since $(x_1-x_2)^2 + 2(x_1 - x_2) \geq 0$ unless $x_1-x_2=-1$, it is easy to see that the inequality holds for $m\geq 1$ and $x_1, x_2$ with $x_1+1\neq x_2$. Now, assume $x_1+1=x_2$; in this case, 
$$ -1 = x_1^2 +(x_1+1)^2  -2x_1^2 - 2x_1 - 2.$$
and the inequality can be reduced to the above \eqref{cond1}.
Hence, the basis is indeed reduced and the length of the shortest vector in the lattice is $\sqrt{1+m^2 + m^4}$.
\end{proof}

\begin{theorem} \label{thm2}
Let $b=2k+1$ be odd and set $N=\frac{b^3-1}{2}$, $\mathbf{v}=(1,b,b^2)$. Then we have
$$ \lim_{k\rightarrow \infty} \frac{\lambda^{\ast}(\Pi_{N,{\mathbf{v}}})}{N^{2/3}} =  2^{1/6} $$
\end{theorem}

\begin{proof}
We prove the theorem by showing that
\begin{equation} \mathbf{v}_1=\begin{pmatrix} -k \\ -k(2k+1) \\ k+k(2k+1) \end{pmatrix}, 
\mathbf{v}_2=\begin{pmatrix} -k(2k+1) \\ k+k(2k+1) \\ -k \end{pmatrix}, 
\mathbf{v}_3=\begin{pmatrix}  k+k(2k+1) +1\\ k+1  \\ k+k(2k+1) +k +1 \end{pmatrix} \end{equation}
is a reduced lattice basis for the lattice
\begin{equation*} \begin{pmatrix} 0 & 0 & 1 \\  0&\frac{(2k+1)^3-1}{2} & (2k+1) \\  \frac{(2k+1)^3-1}{2} & 0&(2k+1)^2 \end{pmatrix}. \end{equation*}
First, we observe that
\begin{equation}
(\mathbf{v}_1, \mathbf{v}_2, \mathbf{v}_3) \cdot 
\begin{pmatrix}  1+k& -k& 1\\ k &1+k &1\\ k&k &1 \end{pmatrix}
= \begin{pmatrix} 0 & 0 & 1 \\  0 &\frac{b^3-1}{2} & b \\ \frac{b^3-1}{2}&0  & b^2 \end{pmatrix}
\end{equation}
And the determinant of the second matrix is indeed 1.
Next we have to check that $\| \mathbf{v}_2 + x_1 \mathbf{v}_1\| \geq \| \mathbf{v}_2\|$. 
We have that $\| \mathbf{v}_2\|^2 = 2k^2 (3 + 6k + 4k^2)$ and the above inequality reduces to
$$ 2k^2 (3 + 6k + 4k^2) (1-x_1+x_1^2) \geq 2k^2 (3 + 6k + 4k^2), $$
which is true for $k\geq 1$ and all integers $x_1$.
Finally, we have to show that
$$\|  \mathbf{v}_3 + x_2  \mathbf{v}_2 + x_1 \mathbf{v}_1 \| \geq \| \mathbf{v}_3 \|$$ 
for all integers $x_1, x_2$.
We have that $\|\mathbf{v}_3 \|^2 =(1+2k+2k^2)(3 + 6k + 4k^2)$.
Setting $\alpha=3 + 6k + 4k^2$, the inequality can be simplified to
$$ (1+2k +2k^2) \alpha + 2k^2 \alpha (x_1^2 + x_2^2 + x_1 - x_2 - x_1 x_2) \geq (1+2k+2k^2) \alpha,$$
and hence to
$$x_1^2 + x_2^2 + x_1 - x_2 - x_1 x_2 \geq 0,$$
Depending on the signs of $x_1$ and $x_2$ 
we either rewrite this expression as $x_1^2 + x_2^2 + x_1- x_2(1+x_1)$ in case $x_1$ and $x_2$ have the same sign
or as $x_1^2 + x_2^2 + x_1(1-x_2) -x_2$ in case $x_1$ and $x_2$ have opposite signs to see that this inequality holds in each case and for all pairs of integers $x_1, x_2$.
Thus, we get that
$$\lim_{k\rightarrow \infty} \frac{\lambda^{\ast}(\Pi_{N,{\mathbf{v}}})}{N^{2/3}} = \lim_{k \rightarrow \infty} \frac{\sqrt{2}\sqrt{k^2(3 + 6k + 4k^2)}}{(k (3 + 6 k + 4 k^2))^{2/3}} = 2^{1/6}.$$
\end{proof}

\begin{figure}
\includegraphics[scale=0.5]{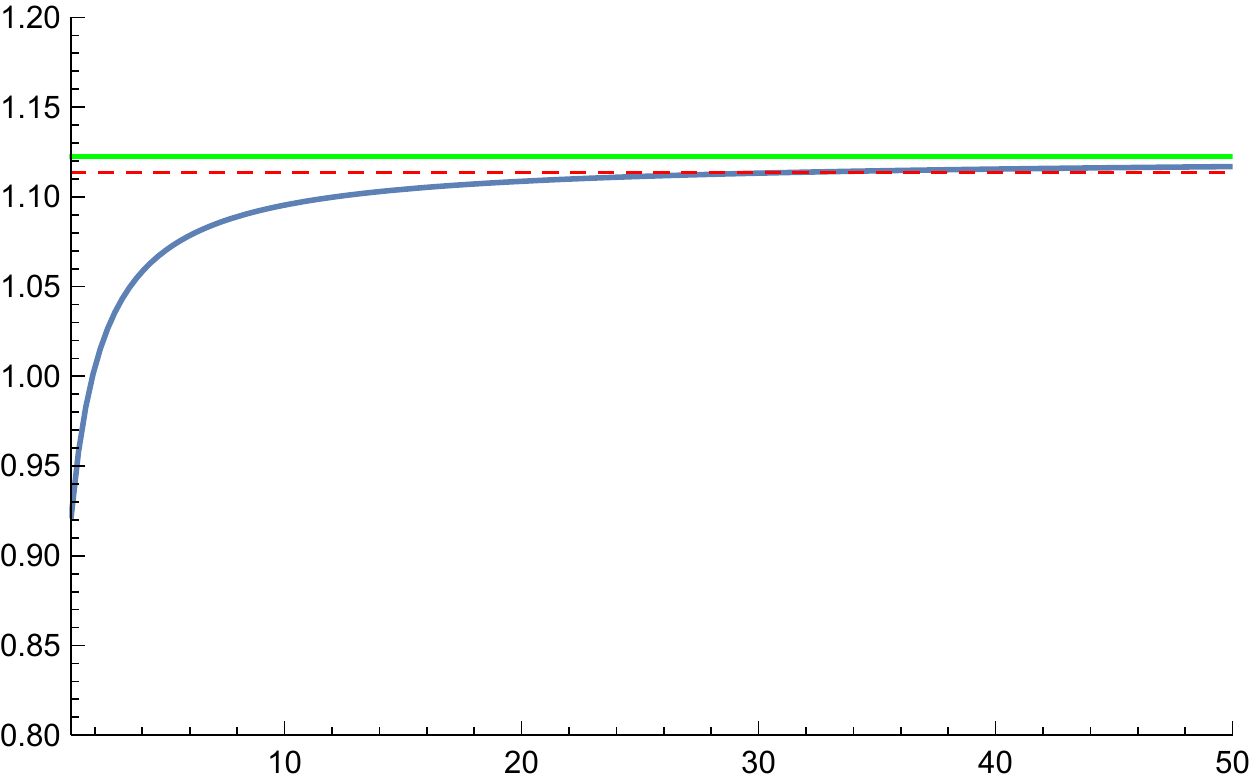} 
\caption{Illustration of the convergence of the normalised length of the shortest vector of $\Pi_{N,\mathbf{v}}$ to $2^{1/6}$. Here, $2N=(2k+1)^3-1$ and the $x$-axis represents $k$; i.e. for $k=50$ we have $N=515150$ and $\mathbf{v}=(1,101,10201)$.
The red dashed line is at $1.1136$; i.e. at the value obtained for $N=244$ with $(1,13,169)$.} \label{fig:convergence}.
\end{figure}

\begin{remark}
We can calculate the value of $k$ needed to get a lattice whose normalised shortest vector is longer than for the exceptional lattice $\Pi=\Pi_{244,(1,13,169)}$. In \eqref{equ1} we calculated that $\lambda^{\ast}(\Pi)/244^{2/3}=\sqrt{1891}/244^{2/3}\approx 1.1136$, which is only improved for $k>31$ or $N\geq 137312$; see Figure \ref{fig:convergence}.
\end{remark}

\section{Higher dimensions}
\label{sec:multiD}

It is possible to generalise our construction to higher dimensions. 
We define $\mlatticeI$ and $\mlatticeII$ analogously to the two- and three-dimensional cases for parameter vectors
$$ \mathbf{v} = (a_1, \ldots, a_d),$$
with $d=4$ or $d=5$, $0<a_1< \ldots < a_d <N$ and $\gcd(N,a_1)=1$. Furthermore, we define
$ \lambda^{\ast}(\mlatticeI)$ and $ \lambda^{\ast}(\mlatticeII)$ as in the three dimensional case.
Using a similar volume argument as in the introduction for the three dimensional case, we get that
$$ \lambda^{\ast}(\mlatticeI) \leq C_4 \cdot N^{3/4} \ \ \ \text{ and } \ \ \ \lambda^{\ast}(\mlatticeII) \leq C_5 \cdot N^{4/5}. $$
Lemma \ref{lem:basis} generalises as well and, therefore, we see that $\covol(\Lambda(\mlatticeI))=N^3$ and $\covol(\Lambda(\mlatticeII))=N^4$. Thus, we get from the definition of the Hermite constant that
\begin{align*} 
\lambda^{\ast}(\mlatticeI) &\leq \covol(\Lambda(\mlatticeI))^{1/4} \cdot \sqrt{\gamma_4}  = N^{3/4} \cdot 2^{1/4}, \\
\lambda^{\ast}(\mlatticeII) &\leq \covol(\Lambda(\mlatticeII))^{1/5} \cdot \sqrt{\gamma_5}  = N^{4/5} \cdot 2^{3/10}. \\
\end{align*}

We can mimic Theorem \ref{thm2} to obtain the following result for $d=4$ and $d=5$.
\begin{theorem} \label{thm3}
Let $b=2k+1$ be odd and set $N=\frac{b^4-1}{2}$, $\mathbf{v}=(1,b,b^2,b^3)$. Then we have
$$ \lim_{k\rightarrow \infty} \frac{\lambda^{\ast}(\mlatticeI)}{N^{3/4}} =  2^{1/4}. $$
Similarly, let $b=2k+1$ be odd and set $N=\frac{b^5-1}{2}$, $\mathbf{v}=(1,b,b^2,b^3,b^4)$. Then we have
$$ \lim_{k\rightarrow \infty} \frac{\lambda^{\ast}(\mlatticeII)}{N^{4/5}} =  2^{3/10}. $$
\end{theorem}

\begin{proof}[Sketch of the proof]
We start with $d=4$. Analogously to the proof of Theorem \ref{thm2} we can show that
\begin{equation} \mathbf{v}_1=\begin{pmatrix} -k \\ -kb \\ -kb^2 \\k+kb+kb^2 \end{pmatrix}, 
\mathbf{v}_2=\begin{pmatrix} -kb \\ -kb^2 \\ k+kb+kb^2 \\ -k \end{pmatrix}, 
\mathbf{v}_3=\begin{pmatrix} -k-kb-kb^2 \\ k \\ kb \\ kb^2 \end{pmatrix}, 
\mathbf{v}_4=\begin{pmatrix}  k+kb+1\\ 2k+kb+kb^2+1  \\ k+kb+1\\ 2k+kb+kb^2+1\end{pmatrix} \end{equation}
is a Minkowski-reduced lattice basis (see \cite[Section 2.2]{stehle} for a definition of reduced basis in more than 3 dimensions and \cite[Theorem 2.2]{stehle} for a list of conditions to check) for the lattice
\begin{equation*} \begin{pmatrix} 0 & 0 & 0 & 1 \\  0& 0& \frac{(2k+1)^4-1}{2} & (2k+1) \\ 0& \frac{(2k+1)^4-1}{2} &0 & (2k+1)^2 \\   \frac{ (2k+1)^4-1}{2} & 0& 0 &(2k+1)^3 \end{pmatrix}. \end{equation*}
From this we get that
$$\lim_{k\rightarrow \infty} \frac{\lambda^{\ast}(\mlatticeI)}{N^{3/4}} = 
\frac{k \sqrt{\left(4 k^2+6 k+3\right) (2 k (k+1)+1)}}{\sqrt{2} (k (k+1) (2 k (k+1)+1))^{3/4}} = 2^{1/4}.$$

Finally, we can construct $5$-dimensional lattices in a similar fashion; i.e. let $b=2k+1$ be odd and set $N=\frac{b^5-1}{2}$, $\mathbf{v}=(1,b,b^2,b^3,b^4)$.
It can be shown that 
{\small \begin{equation*} \mathbf{v}_1=\begin{pmatrix} -k \\ -k b \\ -k b^2 \\ -k b^3 \\ 4k(k + 1)(2k^2 + 2k + 1) \end{pmatrix}, 
\mathbf{v}_2=\begin{pmatrix} k b^2\\ k b^3\\ -4k(k + 1)(2k^2 + 2k + 1)\\ k\\ k b \end{pmatrix}, 
\mathbf{v}_3=\begin{pmatrix} -k b^3 \\ 4k(k + 1)(2k^2 + 2k + 1)\\ -k \\ -k b\\ -k b^2 \end{pmatrix},
\end{equation*}
\begin{equation*}
\mathbf{v}_4=\begin{pmatrix}  -4k(k + 1)(2k^2 + 2k + 1)\\ k\\ k b\\ k b^2\\ k b^3 \end{pmatrix}, 
\mathbf{v}_5=\begin{pmatrix}  (2k^2 + 2k + 1)b\\ (2k^2 + 2k + 1) b^2\\ (4k^2 + 2k + 1)(k + 1)\\ (4k^2 + 2k + 1)(k + 1)b\\ 2k^2 + 2k + 1\end{pmatrix} 
\end{equation*} }
is a Minkowski-reduced lattice basis for the lattice
\begin{equation*} \begin{pmatrix} 0&0 & 0 & 0 & 1 \\  0& 0& 0& \frac{(2k+1)^5-1}{2} & (2k+1) \\ 0&0& \frac{(2k+1)^5-1}{2} &0 & (2k+1)^2 
\\ 0& \frac{(2k+1)^5-1}{2} &0 &0& (2k+1)^3
\\   \frac{ (2k+1)^5-1}{2} & 0& 0 &0 &(2k+1)^4 \end{pmatrix}. \end{equation*}
Having a Minkowski-reduced basis ensures, by a result of van der Waerden \cite{waerden}, that the four shortest vectors in the basis achieve the first four successive minima of the lattice.
Thus, we see that the shortest vector is of the form 
$$ (-k, \ -k b,\ -k b^2,\ -k b^3,\ k + k b + k b^2 + k b^3), $$
such that 
$$\lim_{k\rightarrow \infty} \frac{\lambda^{\ast}(\mlatticeII)}{N^{4/5}} = 2^{3/10}.$$
\end{proof}

\section*{Acknowledgments}
I would like to thank one anonymous reviewer who brought \cite{stehle} to my attention, who inspired the results of Section \ref{sec:multiD} and whose very constructive and detailed feedback greatly improved an earlier version of this manuscript. 


\end{document}